\documentclass[12pt,reqno,oneside]{amsart}

\usepackage{amsthm,amsmath,amssymb}
\usepackage{mathrsfs}
\usepackage{enumerate}
\usepackage{mathtools}
\usepackage{multirow}
\usepackage{enumitem}
\usepackage[symbol]{footmisc}
\usepackage{amsfonts}
\usepackage{verbatim}
\usepackage{amsbsy}
\usepackage{amsmath}
\usepackage{amssymb}
\usepackage{MnSymbol}
\usepackage{mathrsfs}
\usepackage{cite}
\usepackage{algorithm}
\usepackage[noend]{algpseudocode}
\usepackage{float}
\usepackage{multicol}

\newtheorem{theorem}{Theorem}[section]

\newtheorem{proposition}[theorem]{Proposition}

\newtheorem{corollary}[theorem]{Corollary} 

\theoremstyle{definition}
\newtheorem{definition}[theorem]{Definition}

\theoremstyle{remark}
\newtheorem{remark}[theorem]{Remark}
\newtheorem{example}[theorem]{Example}
\newtheorem{case}{Case}
\newtheorem{case2}{Case}

\numberwithin{subcase}{case}
\newtheorem{subcase2}{Case}
\numberwithin{subcase2}{case2}

\numberwithin{subsupposition}{supposition}

\numberwithin{subsupposition2}{supposition2}

\makeatletter
\@addtoreset{case}{theorem}
\makeatother

\makeatletter
\def\BState{\State\hskip-\ALG@thistlm}
\makeatother

\DeclareMathOperator{\pn}{Pn}


\begin{document}
\title{Parikh Matrices for Powers of Words}
\author{Adrian Atanasiu}
\address{Consulting Prof. at Faculty of Mathematics and Computer Science\\
Bucharest University\\
Str. Academiei 14\\
Bucharest 010014, Romania}
\email{aadrian@gmail.com}
\author{Ghajendran Poovanandran}
\address{School of Mathematical Sciences\\
Universiti Sains Malaysia\\
11800 USM, Malaysia}
\email{p.ghajendran@gmail.com}
\author{Wen Chean Teh}
\address{School of Mathematical Sciences\\
Universiti Sains Malaysia\\
11800 USM, Malaysia}
\email[Corresponding author]{dasmenteh@usm.my}

\begin{abstract}
Certain upper triangular matrices, termed as Parikh matrices, are often used in the combinatorial study of words. Given a word, the Parikh matrix of that word elegantly computes the number of occurrences of certain predefined subwords in that word. In this paper, we compute the Parikh matrix of any word raised to an arbitrary power. Furthermore, we propose canonical decompositions of both Parikh matrices and words into normal forms. Finally, given a Parikh matrix, the relation between its normal form and the normal forms of words in the corresponding $M$\!-equivalence class is established.
\end{abstract}

\maketitle
\section{Introduction}

The problem of finding the optimal number of subwords of a word needed to completely determine that word still remains open \cite{jM00}. In the spirit of solving this problem, Mateescu et al. introduced Parikh matrices in \cite{MSSY01} by generalizing the classical Parikh vectors \cite{rP66}. In general, the Parikh matrix of a word is an upper triangular matrix which contains the number of occurrences of certain predefined subwords of that word. Despite storing more information about a word, not every Parikh matrix uniquely determines a word. Nevertheless, Parikh matrices and their variants \cite{AR13,tS04,aC10,CW08,oE04} have opened up the door to various new investigations in the combinatorial study of words (for example, see \cite{aA07,AAP08,AMM02,vS09,SS06,wT14,wT16,wT16b,aS10,SY10,AT16,GT16a,GT17,MSY04,GT17b,BM16,MBS17,SHN09}).

Repetition in words has been intensively studied in the literature and it dates back to the works of Thue in the early 1900s. Often in the literature, a word is expressed as the power of another word; for instance the word \textit{murmur} can be written as $(mur)^2$. In this paper, we deal with such powers of words in relative to Parikh matrices. Our main contributions would be as follows:
\begin{enumerate}[leftmargin=2em]
\item A general formula to obtain the Parikh matrix of any power of a given word;
\item A normal form of an arbitrary Parikh matrix (respectively word) obtained by decomposing that matrix (respectively word) in terms of powers of other Parikh matrices (respectively words).
\end{enumerate}

The remainder of this paper is structured as follows. Section 2 provides the basic terminology and preliminaries. Section 3 deals with Parikh matrices of powers of words. Apart from presenting a general formula to obtain such Parikh matrices, the properties of these matrices are studied as well. In the next section, we propose a normal form of Parikh matrices sustained by a canonical \mbox{decomposition}. An algorithm to obtain this normal form is presented for Parikh matrices over the binary alphabet. Section 5 proposes a normal form of words, analogous to the one for Parikh matrices. The relation between the normal form of an arbitrary Parikh matrix and the normal forms of the words represented by that matrix is then established. Our conclusions follow after that.

\section{Preliminaries}
The set of all positive integers is denoted by $\mathbb{N}$. 

Suppose $\Sigma$ is a finite and nonempty alphabet. The set of all words over $\Sigma$ is denoted by $\Sigma^*$ and $\lambda$ is the unique empty word. Let $\Sigma^+$ denote the set $\Sigma^*\backslash\{\lambda\}$. If $v,w\in\Sigma^*$, the
concatenation of $v$ and $w$ is denoted by $vw$. An ordered alphabet is an alphabet
$\Sigma=\{a_1,a_2,\ldots ,a_s\}$ with an ordering on it. For example, if $a_1<a_2<\cdots <a_s$, then we may write $\Sigma=\{a_1<a_2<\cdots <a_s\}$. For convenience, we shall frequently abuse notation and use $\Sigma$ to denote both the ordered alphabet and its underlying alphabet.

A word $v$ is a \emph{scattered subword} (or simply \textit{subword}) of $w\in \Sigma^*$ if and only if there exist $x_1,x_2,\dotsc, x_n$, $y_0, y_1, \dotsc,y_n\in \Sigma^*$ (possibly empty) such that \mbox{$v=x_1x_2\dotsm x_n \text{ and } w=y_0x_1y_1\dotsm y_{n-1}x_ny_n$}. If the letters in $v$ occur contiguously in $w$ (that is $y_1=y_2=\dotsc=y_{n-1}=\lambda$), then $v$ is a \emph{factor} of $w$.
The number of occurrences of a word $v$ as a subword of $w$ is denoted by $\vert w\vert_v$. 
Two occurrences of $v$ are considered different if and only if they differ by at least one position of some letter. 
For example, $\vert abab\vert_{ab}=3$ and $\vert abcabc\vert_{abc}=4$.
By convention, $\vert w\vert_{\lambda}=1$ for all $w\in \Sigma^*$.
The reader is referred to \cite{RS97} for language theoretic notions not detailed here.

For any integer $n\geq 2$, let $\mathcal{M}_n$ denote the multiplicative monoid of $n\times n$ upper triangular matrices with nonnegative integral entries and unit diagonal. For a matrix $X$, we denote its $(i,j)$-entry by $X_{i,j}$.

\begin{definition} 
Suppose $\Sigma=\{a_1<a_2< \dotsb<a_s\}$ is an ordered alphabet, where $s\ge 2$. The \emph{Parikh matrix mapping} with respect to $\Sigma$, denoted by $\Psi_{\Sigma}$, is the morphism
$$\Psi_{\Sigma}: \Sigma^*\rightarrow \mathcal{M}_{s+1}$$
defined as follows: $\Psi_{\Sigma}(\lambda)=I_{s+1}$; if  $\Psi_{\Sigma}(a_q)=M$, then $M_{i,i}=1$ for each $1\leq i\leq s+1$, $M_{q,q+1}=1$ and all other entries of the matrix $\Psi_{\Sigma}(a_q)$ are zero. Matrices of the form  $\Psi_{\Sigma}(w)$ for $w\in \Sigma^*$ are called \emph{Parikh matrices}. We denote by $\mathcal{P}_{\Sigma}$ the set of all Parikh matrices with respect to $\Sigma$ and let $\mathcal{P}_{\Sigma}^+=\mathcal{P}_{\Sigma}\backslash\{I_{s+1}\}$.
\end{definition}

\begin{theorem}\cite{MSSY01}\label{1206a}
Suppose $\Sigma=\{a_1<a_2< \dotsb<a_s\}$ is an ordered alphabet and $w\in \Sigma^*$. The matrix $\Psi_{\Sigma}(w)=M$ has the following properties:
\begin{itemize}
\item $M_{i,i}=1$ for each $1\leq i \leq s+1$;
\item $M_{i,j}=0$ for each $1\leq j<i\leq s+1$;
\item $M_{i,j+1}=\vert w \vert_{a_ia_{i+1}\dotsm a_j}$ for each $1\leq i\leq j \leq s$.
\end{itemize}
\end{theorem}

\begin{remark}
Suppose $\Sigma=\{a_1<a_2< \dotsb<a_s\}$. The Parikh vector $\Psi(w)=(|w|_{a_1},|w|_{a_2},\ldots,|w|_{a_s})$ of a word $w\in\Sigma^*$ is contained in the second diagonal of the Parikh matrix $\Psi_\Sigma(w)$.
\end{remark}

\begin{example}
Suppose $\Sigma=\{a<b<c\}$ and $w=ababcc$.
Then \begin{align*}
\Psi_{\Sigma}(w)&=\Psi_{\Sigma}(a)\Psi_{\Sigma}(b)\Psi_{\Sigma}(a)\Psi_{\Sigma}(b)\Psi_{\Sigma}(c)\Psi_{\Sigma}(c)\\
&= \begin{pmatrix}
1 & 1 & 0 &0 \\
0 & 1 & 0 & 0\\
0 & 0 & 1 & 0\\
0 & 0 & 0 & 1
\end{pmatrix}
\begin{pmatrix}
1 & 0 & 0 &0 \\
0 & 1 & 1 & 0\\
0 & 0 & 1 & 0\\
0 & 0 & 0 & 1
\end{pmatrix}
\dotsm
\begin{pmatrix}
1 & 0 & 0 &0 \\
0 & 1 & 0 & 0\\
0 & 0 & 1 & 1\\
0 & 0 & 0 & 1
\end{pmatrix}\\
&= \begin{pmatrix}
1 & 2 & 3 & 6 \\
0 & 1 & 2 & 4\\
0 & 0 & 1 & 2\\
0 & 0 & 0 & 1
\end{pmatrix}
=\begin{pmatrix}
1 & \vert w\vert_a & \vert w\vert_{ab} & \vert w\vert_{abc} \\
0 &1 & \vert w\vert_b & \vert w\vert_{bc}\\
0 & 0 & 1 & \vert w\vert_c\\
0 & 0 & 0 & 1
\end{pmatrix}.
\end{align*}
\end{example}

The following is a basic property used to decide whether a matrix in $\mathcal{M}_3$ is a Parikh matrix.
\begin{theorem}(see \cite{MSY04})\label{EntryParikh}
Suppose $M\in\mathcal{M}_3$. The matrix $M$ is a Parikh matrix if and only if $M_{1,3}\le M_{1,2}\cdot M_{2,3}$.
\end{theorem}

\begin{definition}
Suppose $\Sigma$ is an ordered alphabet.
Two words $w,w'\in \Sigma^*$ are \emph{$M$\!-equivalent}, denoted by $w\equiv_Mw'$, iff $\Psi_{\Sigma}(w)=\Psi_{\Sigma}(w')$.
A word $w\in \Sigma^*$ is \emph{$M$\!-ambiguous} iff it is $M$\!-equivalent to another distinct word. Otherwise, $w$ is \emph{$M$\!-unambiguous}. We denote the $M$\!-equivalence class of a word $w\in\Sigma^*$ by $C_w$.
\end{definition}

\section{Powers of Parikh Matrices}

The following result can be used to compute any power of a given matrix in $\mathcal{M}_n$ where integer $n\ge 2$. In particular, since every Parikh matrix is a matrix in $\mathcal{M}_n$ for some integer $n\ge 2$, this result can be applied to it as well.

\begin{theorem}\label{PowerTheorem}
For every integer $m\geq 1$, $n\geq 2$, and $X \in \mathcal{M}_n$,
\begin{equation*}
\displaystyle
(X^m)_{i,j}= \begin{cases}
\sum\limits_{t=1}^{j-i} \binom{m}{t} \underbrace{\sum\limits_{i<k_1<k_2<\dotsb <k_{t-1}<j} X_{i,k_1}X_{k_1,k_2}\dotsm X_{k_{t-1},j} }_{\text{understood to be } X_{i,j} \text{ when } t=1} &\text{ if } 1\leq i<j\leq n,\\
1 &\text{ if } 1\leq i=j \leq n,\\
0 & \text{ if } 1\leq j<i\leq n.
\end{cases}
\end{equation*}
\end{theorem}

\begin{proof}
We prove by induction on the power $m$. The base step is obvious. For the induction step, 
we only consider the case $i<j$ as the other two cases trivially hold.
We have
\begin{align}
(X^{m+1})_{i,j}={}&\sum_{l=1}^{n} (X^m)_{i,l} X_{l,j} \nonumber\\
={}& \sum_{l=i}^{j} (X^m)_{i,l} X_{l,j} \nonumber\\
={}&X_{i,j}  + \sum_{l=i+1}^{j-1} (X^m)_{i,l} X_{l,j}+ (X^{m})_{i,j }. \tag{$\ast$}
\end{align}
Now,
\begin{align*}
\sum_{l=i+1}^{j-1} (X^m)_{i,l} X_{l,j}
={}&\sum_{l=i+1}^{j-1} \left[\sum\limits_{t=1}^{l-i} \binom{m}{t} \sum_{i<k_1<k_2<\dotsb <k_{t-1}<l} X_{i ,k_1}X_{k_1,k_2}\dotsm X_{k_{t-1},l}\right] X_{l,j}\\
={}&\sum_{l=i+1}^{j-1} \sum\limits_{t=1}^{l-i} \left[\binom{m}{t} \sum_{i<k_1<k_2<\dotsb <k_{t-1}<l} X_{i ,k_1}X_{k_1,k_2}\dotsm X_{k_{t-1},l} X_{l,j}\right]\\
={}&\sum_{t=1}^{j-i-1} \sum_{l=i+t}^{j-1} \left[\binom{m}{t} \sum_{i<k_1<k_2<\dotsb <k_{t-1}<l} X_{i ,k_1}X_{k_1,k_2}\dotsm X_{k_{t-1},l} X_{l,j}\right]\\
={}&\sum_{t=1}^{j-i-1} \binom{m}{t}  \sum_{l=i+t}^{j-1} \; \sum_{i<k_1<k_2<\dotsb <k_{t-1}<l} X_{i ,k_1}X_{k_1,k_2}\dotsm X_{k_{t-1},l} X_{l,j}\\
={}&\sum_{t=1}^{j-i-1} \binom{m}{t}  \sum_{i<k_1<k_2<\dotsb <k_{t-1}<l<j} X_{i,k_1}X_{k_1,k_2}\dotsm X_{k_{t-1},l} X_{l,j}\\
={}&\sum_{t=1}^{j-i-1} \binom{m}{t}  \sum_{i<k_1<k_2<\dotsb <k_{t}<j} X_{i,k_1}X_{k_1,k_2}\dotsm X_{k_{t},j} \\
={}&\sum_{t=2}^{j-i} \binom{m}{t-1}  \sum_{i<k_1<k_2<\dotsb <k_{t-1}<j} X_{i,k_1}X_{k_1,k_2}\dotsm X_{k_{t-1},j}\\
={}&-X_{i,j}+ \sum_{t=1}^{j-i} \binom{m}{t-1}  \sum_{i<k_1<k_2<\dotsb <k_{t-1}<j} X_{i,k_1}X_{k_1,k_2}\dotsm X_{k_{t-1},j}. \tag{$\ast\ast$}
\end{align*}
(The third equality is obtained by interchanging the order of summation.) Since $\binom{m+1}{t}=\binom{m}{t-1}+\binom{m}{t}$, the induction step is complete because by combining $(\ast)$ and $(\ast\ast)$, we have
$$(X^{m+1})_{i,j}=\sum_{t=1}^{j-i} \left[\binom{m}{t-1}+\binom{m}{t} \right]   \sum_{i<k_1<k_2<\dotsb <k_{t-1}<j} X_{i ,k_1}X_{k_1,k_2}\dotsm X_{k_{t-1},j}.$$
\end{proof}

Since every Parikh matrix is in $\mathcal{M}_n$ for some integer $n\ge 3$, the Parikh matrix of a word to the power of $m$ (where $m$ is a positive integer) can be computed by Theorem~\ref{PowerTheorem}. The following example illustrates this.

\begin{example}
Consider the word $abb$ in $\{a<b\}^*$. Suppose $m$ is a positive integer. Then the Parikh matrix of the word $(abb)^m$ can be computed as follows:
\begin{align*}
\Psi_{\Sigma}((abb)^m)=(\Psi_{\Sigma}(abb))^m &=\begin{pmatrix}
1 & 1 & 2\\ 
0 & 1 & 2\\ 
0 & 0 & 1
\end{pmatrix}^m\\
&=\begin{pmatrix}
1 & m\cdot 1 & m\cdot 2+\binom{m}{2}\cdot 1\cdot 2\\ 
0 & 1 & m\cdot 2\\ 
0 & 0 & 1
\end{pmatrix}\\
&=\begin{pmatrix}
1 & m & m^2+m\\ 
0 & 1 & 2m\\ 
0 & 0 & 1
\end{pmatrix}.
\end{align*}
\end{example}

\begin{definition}
Suppose $m$ and $n$ are positive integers such that $n\ge 2$. We define the function \mbox{$f_m:\mathcal{M}_n\rightarrow\mathcal{M}_n$} by $f_m(X)=X^m$ for all $X\in\mathcal{M}_n$.
\end{definition}

If $X$ is a Parikh matrix, then clearly $f_m(X)$ is a Parikh matrix as well. However, the converse is not necessarily true. In fact, the following is a consequence of Theorem~\ref{EntryParikh} and Theorem~\ref{PowerTheorem} for the binary alphabet which can be used to determine whether $f_m(X)$ is a Parikh matrix. 

\begin{proposition}\label{CharBinParMat}
Suppose $X\in\mathcal{M}_3$ and $m$ is a positive integer. Let 
$X=\begin{pmatrix}
1 & a & c\\ 
0 & 1 & b\\ 
0 & 0 & 1
\end{pmatrix}$. The matrix $X^m$ is a Parikh matrix if and only if either of the following holds: 
\begin{enumerate}
\item if either $a$ or $b$ is zero, then $c=0$;
\item otherwise if both $a$ and $b$ are nonzero, then $\dfrac{c}{ab}\le \dfrac{m+1}{2}$.
\end{enumerate}

\end{proposition}
\begin{proof}
The biconditional holds trivially for (1), thus it remains to show (2). By Theorem~\ref{PowerTheorem}, we have $X^m=\begin{pmatrix}
1 & ma & mc+\dfrac{m(m-1)}{2}ab\\ 
0 & 1 & mb\\ 
0 & 0 & 1
\end{pmatrix}$. By Theorem~\ref{EntryParikh}, the matrix $X^m$ is a Parikh matrix if and only if $$mc+\dfrac{m(m-1)}{2}ab\le ma\cdot mb.$$ The above inequality can be reduced to $\dfrac{c}{ab}\le \dfrac{m+1}{2}$, thus the conclusion holds.
\end{proof}

The following is immediate by Proposition~\ref{CharBinParMat}.

\begin{corollary}
Suppose $\Sigma$ is an ordered alphabet with $|\Sigma|=2$. For every matrix $X\in\mathcal{M}_3$ with nonzero entries above the main diagonal, there exists a positive integer $M$ such that $X^m\in\mathcal{P}_\Sigma^+$ for all integers $m\ge M$.
\end{corollary}

\begin{example}
Let $\Sigma=\{a<b\}$. Consider the matrix $X=\begin{pmatrix}
1 & 2 & 9\\ 
0 & 1 & 2\\ 
0 & 0 & 1
\end{pmatrix}$. Then $X^4=\begin{pmatrix}
1 & 8 & 60\\ 
0 & 1 & 8\\ 
0 & 0 & 1
\end{pmatrix}\in\mathcal{P}_\Sigma$. It can easily be checked by using Theorem~\ref{EntryParikh} or Proposition~\ref{CharBinParMat} that $X^{m}\not\in\mathcal{P}_\Sigma$ for all integers $1\le m<4$. An example of word $w\in\Sigma^*$ with $\Psi_\Sigma(w)=X^4$ is $w=a^7b^4ab^4$.
\end{example}

The following result shows that for every positive integer $m$, the function $f_m$ is injective.

\begin{theorem}\label{InjectPowFunc}
Suppose $m$ is a positive integer and $X,Y \in \mathcal{M}_n$ for some integer $n\ge 2$. If $X^m=Y^m$, then $X=Y$.
\end{theorem}

\begin{proof}
Suppose $X^m=Y^m$.
We prove by strong induction that the second diagonal, the third diagonal and so forth of $X$ and $Y$ are equal, thus $X=Y$. For the base step (corresponding to the second diagonal), we need to show that $X_{i,j}=Y_{i,j}$ 
whenever $j-i=1$. (It is understood that $1\leq i$ and $j\leq n$ must hold.)
Fix $1\leq i\leq n-1$. By Theorem~\ref{PowerTheorem}, $(X^m)_{i,i+1}= m X_{i,i+1} $ and $(Y^m)_{i,i+1}= m Y_{i,i+1} $. Since $X^m=Y^m$, it follows that $X_{i,i+1}=Y_{i,i+1}$.

For the induction step, we need to show that $X_{i,j}=Y_{i,j}$ holds whenever \mbox{$j-i=N+1$}, assuming that $(X)_{i,j'}=(Y)_{i,j'}$ whenever $1\leq j'-i\leq N$. Fix $1\leq i\leq n-N-1$ and let $j=i+N+1$.
By Theorem~\ref{PowerTheorem},
$$(X^m)_{i,j}= \sum\limits_{t=1}^{j-i} \binom{m}{t} \sum_{i<k_1<k_2<\dotsb <k_{t-1}<j} X_{i,k_1}X_{k_1,k_2}\dotsm X_{k_{t-1},j} .$$
Therefore,
\begin{align*}
X_{i,j}={}& (X^m)_{i,j}-\sum\limits_{t=2}^{j-i} \binom{m}{t} \sum_{i<k_1<k_2<\dotsb <k_{t-1}<j} X_{i ,k_1}X_{k_1,k_2}\dotsm X_{k_{t-1},j}\\
={}& (Y^m)_{i,j}-\sum\limits_{t=2}^{j-i} \binom{m}{t} \sum_{i<k_1<k_2<\dotsb <k_{t-1}<j} Y_{i,k_1}Y_{k_1,k_2}\dotsm Y_{k_{t-1},j}=Y_{i,j},
\end{align*}
thus the proof is complete. (Note that the last equality holds by our induction hypothesis and the assumption that $X^m=Y^m$.) 
\end{proof}

\begin{corollary}
Suppose $\Sigma$ is an ordered alphabet with $|\Sigma|\ge 2$ and $v,w\in\Sigma^*$. Then either of the following holds:
\begin{enumerate}
\item $v^m\equiv_M w^m$ for all positive integers $m$; 
\item $v^m\not\equiv_M w^m$ for all positive integers $m$.
\end{enumerate}
\end{corollary}

\begin{proof}
If $v\equiv_M w$, it follows trivially that
$v^m\equiv_M w^m$ for all integers $m$. Hence, it suffices to prove that
if there exists an integer $m\ge 2$ such that $v^m\equiv_M w^m$, then $v\equiv_M w$.
Suppose $v^m\equiv_M w^m$ for some integer $m\geq 2$. 
Then $(\Psi_{\Sigma}(v))^m=\Psi_{\Sigma}(v^m)=\Psi_{\Sigma}(w^m)=(\Psi_{\Sigma}(w))^m$.
Since $\Psi_{\Sigma}(v),\Psi_{\Sigma}(w)\in \mathcal{M}_{|\Sigma|+1}$, by Theorem~\ref{InjectPowFunc},
we have $\Psi_{\Sigma}(v)=\Psi_{\Sigma}(w)$ and thus $v\equiv_M w$.
\end{proof}

We end this section by the following observation on the $M$\!-equivalence class of an arbitrary power of any word.

\begin{proposition}\label{EqualityMequivClass}
Suppose $\Sigma$ is an ordered alphabet with $|\Sigma|\ge 2$ and $w\in\Sigma^*$. For every positive integer $m$, we have $|C_{w^m}|\ge |C_w|^m$.
\end{proposition}

\begin{proof}
Fix a positive integer $m$. If $w_i\equiv_M w$ for all integers $1\le i\le m$, then $w^m=\underbrace{www\cdots w}_\text{$m$ times} \equiv_M w_1w_2\cdots w_m$. Thus, $|C_{w^m}|\ge\underbrace{|C_w||C_w|\,\cdots\, |C_w|}_\text{$m$ times}=|C_w|^m$.
\end{proof}

\begin{remark}
Suppose $\Sigma$ is an ordered alphabet with $|\Sigma|=2$, $w\in\Sigma^+$ and $m$ is a positive integer. If $|C_{w^m}|=|C_w|^m$, then $|C_w|=1$. The converse however does not hold. For instance, let $w=aba$ (clearly, $|C_w|=1$). Then, $C_{w^2}=\{abaaba,aabbaa,baaaab\}$, therefore $|C_{w^2}|=3$.
\end{remark}

\section{A Normal Form of Parikh Matrices}

Suppose $\Sigma$ is an ordered alphabet. In this section, given a Parikh matrix $M\in\mathcal{P}_{\Sigma}$, we aim to decompose $M$ into a product of some other Parikh matrices, each raised to a \mbox{certain} power. For Parikh matrices with entries large enough, the following decomposition is interesting.


\begin{definition}\label{DefNormalForm}
Suppose $\Sigma$ is an ordered alphabet with $|\Sigma|=s$ and $M\in\mathcal{P}_{\Sigma}^+$. 

\vspace{0.7em}\item \begin{itemize}[leftmargin=1em]

\item Define $\mu(M)=\max\{\, n\in\mathbb{N}\,\,|\, M=A\cdot B^n \text{ for some }A\in\mathcal{P}_{\Sigma} \text{ and } B\in\mathcal{P}_{\Sigma}^+\,\}$.

\item Define $\sigma(M)$ to be the sum of the entries in the second diagonal of $M$.

\item Define $\vartheta(M)$ as follows:
\begin{itemize}[leftmargin=1.5em]
\item if $\mu(M)=1$, then 
$\vartheta(M)$ is defined to be the minimum element of the following set:
$$\{\,\sigma(B)\,\,|\,\,B\in\mathcal{P}_{\Sigma}^+ \text{ and }
M=A\cdot B \text{ for some } A\in\mathcal{P}_\Sigma^+\}\text{ with }\mu(A)\neq 1\},  $$
provided it is nonempty; otherwise, it is defined to be $\sigma(M)$;
\item if $\mu(M)>1$, then  $\vartheta(M)$ is defined to be the maximum element of the following set:
$$\{\,\sigma(B)\,\,|\,\,  B\in\mathcal{P}_{\Sigma}^+ \text{ and }
M=A\cdot B^{\mu(M)} \text{ for some } A\in\mathcal{P}_\Sigma\}.$$
\end{itemize}
\item Define $S_M=\{\, (A,B,\mu(M))\,\,|\, A\in\mathcal{P}_\Sigma \text{ and } B\in\mathcal{P}_{\Sigma}^+   \text{ with } M=A\cdot B^{\mu(M)}\text{ and}$ $\sigma(B)=\vartheta(M)\,\}$. 
\end{itemize}

\vspace{0.7em}\noindent Let $k$ be a nonnegative integer. For every integer $0\le i\le k$, suppose $B_i\in\mathcal{P}_\Sigma$ and $n_i\in\mathbb{N}$. We say that $B_{k}^{n_{k}}B_{k-1}^{n_{k-1}}\cdots B_0^{n_0}$ is a \textit{rl-Parikh normal form} of $M$ if and only if the following holds:
\begin{quote}
\centering
Let $A_0=M$, $A_i=B_{k}^{n_{k}}B_{k-1}^{n_{k-1}}\cdots B_i^{n_i}$ ($1\le i\le k$) and $A_{k+1}=I_{s+1}$.\\
Then $(A_{i+1},B_i,n_i)\in S_{A_i}$ for all $0\le i\le k$.
\end{quote}
Equivalently, we say that $M$ is \textit{rl-Parikh normalized} to the form $B_{k}^{n_{k}}B_{k-1}^{n_{k-1}}\cdots B_0^{n_0}$.
\end{definition}

\begin{remark}\label{RemarkCondDefMat}
The requirement $B\in\mathcal{P}_{\Sigma}^+$ in the first item of Definition~\ref{DefNormalForm} eliminates the trivial decomposition of a Parikh matrix $M$ into $M=M\cdot I_{s+1}^n$ at each stage as $n$ does not have an upper bound in this case.
\end{remark}

\begin{remark}\label{RemChopOffMat}
Suppose $\Sigma$ is an ordered alphabet and $M\in\mathcal{P}_\Sigma^+$.
Let $B_{k}^{n_{k}}B_{k-1}^{n_{k-1}}\cdots B_0^{n_0}$ be a $rl$-Parikh normal form of $M$. For any integer $0\le i\le k$, the form $B_{k}^{n_{k}}B_{k-1}^{n_{k-1}}\cdots B_i^{n_i}$ is a $rl$-Parikh normal form of the matrix $B_{k}^{n_{k}}B_{k-1}^{n_{k-1}}\cdots B_i^{n_i}$.
\end{remark}

\begin{remark}\label{RemHighPowMat}
Suppose $\Sigma$ is an ordered alphabet and $M\in\mathcal{P}_\Sigma^+$. If $M=A\cdot B^n$ for some $A\in\mathcal{P}_\Sigma$, $B\in\mathcal{P}_\Sigma^+$ and positive integer $n$, then $\mu(M)\ge n$.
\end{remark}

One can see that the \textit{rl}-Parikh normal form of a Parikh matrix is not necessarily unique. For a trivial example, let $\Sigma=\{a<b\}$ and consider the word $w=abba$. Then, the matrix $\Psi_\Sigma(w)$ has two \textit{rl}-Parikh normal forms, which are $\Psi_\Sigma(a)[\Psi_\Sigma(b)]^2\Psi_\Sigma(a)$ and $\Psi_\Sigma(b)[\Psi_\Sigma(a)]^2\Psi_\Sigma(b)$. 

The following is a feasible approach to find the Parikh normal form(s) of a Parikh matrix for the binary alphabet. (Here, we are only interested in ``nontrivial'' cases where both entries in the second diagonal are nonzero.)

At each stage of decomposition, given a Parikh matrix $M=\begin{pmatrix}
1 & u & t\\ 
0 & 1 & v\\ 
0 & 0 & 1
\end{pmatrix}$ with integers $u,v>0$, we aim to find two other Parikh matrices $A=\begin{pmatrix}
1 & p & r\\ 
0 & 1 & q\\ 
0 & 0 & 1
\end{pmatrix}$ and 
$B=\begin{pmatrix}
1 & x & z\\ 
0 & 1 & y\\ 
0 & 0 & 1
\end{pmatrix}$ such that for some positive integer $n$, we have $M=A\cdot B^n$ where $(A,B,n)\in S_M$.

By Theorem~\ref{PowerTheorem}, we have $B^n=\begin{pmatrix}
1 & nx & nz+\binom{n}{2}xy\\ 
0 & 1 & ny\\ 
0 & 0 & 1
\end{pmatrix}$, thus it follows that $A\cdot B^n=\begin{pmatrix}
1 & p+nx & r+nz+npy+\binom{n}{2}xy\\ 
0 & 1 & q+ny\\ 
0 & 0 & 1
\end{pmatrix}$. Since $M=A\cdot B^n$, the following system holds:
\begin{center}
$\begin{cases}
p+nx=u,\\
q+ny=v,\\
r+nz+npy+\binom{n}{2}xy=t.\\
\end{cases}$
\end{center}
Furthermore, by Theorem~\ref{EntryParikh}, we have 
\begin{center}
$r\le pq, z\le xy$.
\end{center}
We propose the following algorithm to find the solution to the above system.

\begin{algorithm}[H]
\caption{Decomposition of a Parikh Matrix $M$ into $A\cdot B^n$ where $n$ is maximal (for the binary alphabet)}\label{NormalAlgo} 
\begin{algorithmic}[1]
\State \textit{begin}
\State $n \gets \max\{u,v\}$

\State $Z \gets \{\}$

\State $X \gets \{(x,y)\,|\,x,y>0,\,u-nx\ge 0,\,v-ny\ge 0\}$

\While{$X\neq\{\}$}
    \State \textbf{choose} $(x,y)\in X$
    \State $X \gets X\backslash\{(x,y)\}$
    \State $p \gets u-nx$
    \State $q \gets v-ny$
    \State $Y \gets\{(r,z)\,|\,0\le r\le pq,\,0\le z\le xy,\,r+npy+nz+\binom{n}{2}xy=t\}$
    \If{$Y\neq\{\}$} 
        \For{every $(r,z)\in Y$}
        \State $Z \gets Z\cup\{(n,p,q,r,x,y,z)\}$
        \EndFor
    \EndIf
\EndWhile 
  
\If {$Z=\{\}$}
\State $n \gets n-1$
\State \textbf{goto} 4

\Else

  \State \textbf{return} $Z$

\EndIf
\State \textit{end}
\end{algorithmic}
\end{algorithm}

Clearly, each $z\in Z$ corresponds to some $A\in\mathcal{P}_\Sigma$, $B\in\mathcal{P}_\Sigma^+$ and positive integer $n$ such that $M=A\cdot B^n$ and $n=\mu(M)$. It remains to choose the triplet(s) $(A,B,n)$ satisfying the equality $\sigma(B)=\vartheta(M)$ (see Definition~\ref{DefNormalForm}).

\begin{example}\label{ExNormalMat}
The Parikh matrix $$M=\begin{pmatrix}
1 & 8 & 16\\ 
0 & 1 & 3\\ 
0 & 0 & 1
\end{pmatrix}$$ 
has the following Parikh normal forms: 
\vspace{0.4em}
\begin{enumerate}
\item $\begin{pmatrix}
1 & 1 & 0\\ 
0 & 1 & 0\\ 
0 & 0 & 1
\end{pmatrix}^4\cdot 
\begin{pmatrix}
1 & 0 & 0\\ 
0 & 1 & 1\\ 
0 & 0 & 1
\end{pmatrix}\cdot 
\begin{pmatrix}
1 & 2 & 1\\ 
0 & 1 & 1\\ 
0 & 0 & 1
\end{pmatrix}^2;$  

\item $\begin{pmatrix}
1 & 1 & 0\\ 
0 & 1 & 0\\ 
0 & 0 & 1
\end{pmatrix}^2\cdot 
\begin{pmatrix}
1 & 0 & 0\\ 
0 & 1 & 1\\ 
0 & 0 & 1
\end{pmatrix}\cdot 
\begin{pmatrix}
1 & 1 & 0\\ 
0 & 1 & 0\\ 
0 & 0 & 1
\end{pmatrix}^2\cdot
\begin{pmatrix}
1 & 2 & 2\\ 
0 & 1 & 1\\ 
0 & 0 & 1
\end{pmatrix}^2.$  
\end{enumerate}
\end{example}

\begin{theorem}\label{UniqueParikhNormalForm}
Suppose $\Sigma$ is an ordered alphabet. If $w\in\Sigma^*$ is $M$\!-unambiguous, then the \textit{rl}-Parikh normal form of $\Psi_\Sigma(w)$ is unique.
\end{theorem}
\begin{proof}
Suppose $w$ is $M$\!-unambiguous and let $\Psi_\Sigma(w)=M$. We argue by contradiction. Assume there exist two distinct \textit{rl}-Parikh normal forms of $M$; let them be $B_{k}^{n_{k}}B_{k-1}^{n_{k-1}}\cdots B_0^{n_0}$ and $C_{j}^{m_{kj}}C_{j-1}^{m_{j-1}}\cdots C_0^{m_0}$ respectively. Since they are distinct, it follows that there exists an integer $0\le l\le \min\{j,k\}$ such that 
\begin{enumerate}
\item $n_i=m_i$ and $B_i=C_i$ for all integers $0\le i\le l-1$;
\item $n_l\neq m_l$ or $B_l\neq C_l$.
\end{enumerate}

Let $A\!=\!B_{l-1}^{n_{l-1}}B_{l-2}^{n_{l-2}}\cdots B_0^{n_0}\!=C_{l-1}^{m_{l-1}}C_{l-2}^{m_{l-2}}\cdots C_0^{m_0}$, $B'\!=\!B_{k}^{n_{k}}B_{k-1}^{n_{k-1}}\cdots B_l^{n_l}$ and $C'=C_{j}^{m_{kj}}C_{j-1}^{m_{j-1}}\cdots C_l^{m_l}$.
Then, we have $B'\cdot A=M=C'\cdot A$.
Since Parikh matrices are invertible, it follows that $B'=C'$. By Remark~\ref{RemChopOffMat}, it holds that $n_l=\mu(B')=\mu(C')=m_l$ and $\sigma(B_l)=\vartheta(B')=\vartheta(C')=\sigma(C_l)$. Since $n_l=m_l$, by (2), it must be the case that $B_l\neq C_l$.

Let $v,v'\in\Sigma$ be such that $\Psi_\Sigma(v)=B_l$ and $\Psi_\Sigma(v')=C_l$. Note that $v\neq v'$ because $B_l\neq C_l$. Also, $|v|=|v'|$ because $\sigma(B_l)=\sigma(C_l)$. Let $u,u',y\in\Sigma^*$ be such that $\Psi_\Sigma(u)=B_{k}^{n_{k}}B_{k-1}^{n_{k-1}}\cdots B_{l+1}^{n_{l+1}}$,  $\Psi_\Sigma(u')=C_{j}^{m_{j}}C_{j-1}^{m_{j-1}}\cdots C_{l+1}^{m_{l+1}}$ and $\Psi_\Sigma(y)=A$. Then, $\Psi_\Sigma(uvy)=B'A=M=C'A=\Psi_\Sigma(u'v'y)$. Since $|v|=|v'|$ but $v\neq v'$, it follows that $uvy$ and $u'v'y$ are distinct words. However, this gives us a contradiction as $w$ is $M$\!-unambiguous. Thus our conclusion holds.
\end{proof}

Our final result in this section is a characterization of the following class of Parikh matrices. One can see that this class of Parikh matrices arises naturally by Definition~\ref{DefNormalForm}. 

\begin{definition}\label{DefPrimitive}
Suppose $\Sigma$ is an ordered alphabet and $M\in\mathcal{P}_\Sigma$. We say that $M$ is a \textit{primitive Parikh matrix} if and only if the only \textit{rl}-Parikh normal form of $M$ is $M$ itself.
\end{definition}

\begin{theorem}
Suppose $\Sigma$ is an ordered alphabet and $M\in\mathcal{P}_\Sigma$. The matrix $M$ is a primitive Parikh matrix if and only if every $w\in\Sigma^*$ with $\Psi_\Sigma(w)=M$ is square-free.
\end{theorem}
\begin{proof}
This is straightforward by Definition~\ref{DefNormalForm} and Definition~\ref{DefPrimitive}.
\end{proof}

\section{A Normal Form of Words}

In this section, we introduce a notion analogous to the one in Section 4 - in the perspective of words.


\begin{definition}\label{DefNormalFormWord}
Suppose $\Sigma$ is an alphabet and $w\in\Sigma^+$. 

\begin{itemize}[leftmargin=1em]
\item Define $R_w=\{\,(u,v,n)\in\Sigma^*\times\Sigma^+\times \mathbb{N}\,\,|\,\,w=uv^n\,\}$.

\item Define $\tau(w)=\max\{\,n\in\mathbb{N}\,\,|\,(u,v,n)\in R_w \text{ for some }u\in\Sigma^* \text{ and } v\in\Sigma^+\,\}$.

\item Define $\theta(w)$ as follows:
\begin{itemize}
\item if $\tau(w)=1$, then $\theta(w)$ is defined to be the minimum element of the following set:
$$\{\,|v|\,\,|\,\,v\in\Sigma^+ \text{ and }
w=uv \text{ for some } u\in\Sigma^+\text{ with }\tau(u)\neq 1\},$$
provided it is nonempty; otherwise $\theta(w)=|w|$.
\item if $\tau(w)>1$, then $\theta(w)$ is defined to be the maximum element of the following set:
$$\{\,|v|\,\,|\,\,v\in\Sigma^+ \text{ and }
w=uv^{\tau(w)}\text{ for some } u\in\Sigma^+\}.$$
\end{itemize}

\item Define $\rho(w)=(u',v',\tau(w))$ to be the unique triplet in $R_w$ such that $|v'|=\theta(w)$.

\item Let $w_0=w$ and $(w_1,v_0,n_0)=\rho(w_0)$. For all integers $i\ge 1$ and while $w_i\neq \lambda$, recursively define $(w_{i+1},v_i,n_i)=\rho(w_i)$. Let $k\ge 1$ be the largest integer such that $w_k\neq \lambda$. 
\end{itemize}
We say that $v_{k}^{n_{k}}v_{k-1}^{n_{k-1}}\cdots v_0^{n_0}$ is the \textit{rl-Parikh normal form} of $w$, denoted by $\pn_r(w)$. Equivalently, we say that $w$ is \textit{rl-Parikh normalized} to the form $v_{k}^{n_{k}}v_{k-1}^{n_{k-1}}\cdots v_0^{n_0}$.
\end{definition}

\begin{remark}
The requirement $v\in\Sigma^+$ in the first item of Definition~\ref{DefNormalFormWord} eliminates the trivial decomposition of a word $w$ into $w=w\cdot \lambda^n$ at each stage as $n$ does not have an upper bound in this case.
\end{remark}

\begin{remark}\label{RemChopOffWord}
Suppose $\Sigma$ is an ordered alphabet and $w\in\Sigma^+$.
Let $\pn_r(w)=v_{k}^{n_{k}}v_{k-1}^{n_{k-1}}\cdots v_0^{n_0}$. For any integer $0\le i\le k$, $\pn_r(v_{k}^{n_{k}}v_{k-1}^{n_{k-1}}\cdots v_i^{n_i})=v_{k}^{n_{k}}v_{k-1}^{n_{k-1}}\cdots v_i^{n_i}$.
\end{remark}

\begin{remark}\label{RemHighPowWord}
Suppose $\Sigma$ is an ordered alphabet and $w\in\Sigma^+$. If $w=uv^n$ for some $u\in\Sigma^*$, $v\in\Sigma^+$ and positive integer $n$, then $\tau(w)\ge n$.
\end{remark}

\begin{example}
Suppose $\Sigma=\{a,b,c\}$. Then, we have $\pn_r(bbabbabba)=(bba)^3$, $\pn_r(acccabab)=ac^3(ab)^2$ and $\pn_r(cbcbbaabaaba)=(cb)^2ba(aba)^2$. In the last case, it is understood that the $rl$-Parikh normal form of the word $cbcbbaabaaba$ is $(cb)^2(ba)^1(aba)^2$ and not $(cb)^2b^1a^1(aba)^2$.
\end{example}

The next theorem establishes a significant relation between the \textit{rl}-Parikh normal form of a word and the \textit{rl}-Parikh normal form(s) of the Parikh matrix corresponding to that word.

\begin{definition}\label{DefRelationWord}
Suppose $\Sigma$ is an alphabet and $w,w'\in\Sigma^*$ are distinct words such that $w$ and $w'$ are $M$\!-equivalent. Let $\pn_r(w)=v_{k}^{n_{k}}v_{k-1}^{n_{k-1}}\cdots v_0^{n_0}$ and $\pn_r(w')={y}_j^{m_j}{y}_{j-1}^{m_{j-1}}\cdots {y}_0^{m_0}$.
We write $w\prec w'$ if and only if there exists an integer $0\le N\le \min\{j,k\}$ such that 
\begin{enumerate}
\item $n_i=m_i$ and $v_i=y_i$ for all integers $0\le i\le N-1$; and 
\item either of the following holds:
\begin{itemize}
\item[(i)] $n_N<m_N$;
\item[(ii)] $n_N=m_N=1$ and $|v_N|>|y_N|$;
\item[(iii)] $n_N=m_N>1$ and $|v_N|<|y_N|$.
\end{itemize}
\end{enumerate}
\end{definition}

\begin{definition}
Suppose $\Sigma$ is an ordered alphabet. We say that the word $w$ is \textit{maximal} with respect to the relation $\prec$ (or simply \textit{$\prec$-maximal}), if and only if there exists no other word $w'\in C_w$ such that $w\prec w'$.
\end{definition}

\begin{theorem}\label{RelNormalMatWord}
Suppose $\Sigma$ is an ordered alphabet and $w\in\Sigma^*$. Let $\pn_r(w)=v_{k}^{n_{k}}v_{k-1}^{n_{k-1}}\cdots v_0^{n_0}$. If $w$ is $\prec$-maximal, then $[\Psi_\Sigma(v_{k})]^{n_{k}}[\Psi_\Sigma(v_{k-1})]^{n_{k-1}}\cdots [\Psi_\Sigma(v_0)]^{n_0}$ is an \textit{rl}-Parikh normal form of $\Psi_\Sigma(w)$.
\end{theorem}
\begin{proof}
(The notations used here follow from Definition~\ref{DefNormalForm} and Definition~\ref{DefNormalFormWord}.)

Suppose $w$ is $\prec$-maximal. Let $A_0=\Psi_\Sigma(w)$, $A_i=[\Psi_\Sigma(v_{k})]^{n_{k}}[\Psi_\Sigma(v_{k-1})]^{n_{k-1}}\cdots$ $[\Psi_\Sigma(v_i)]^{n_i}$ ($1\le i\le k$) and $A_{k+1}=I_{s+1}$. By Definition~\ref{DefNormalForm}, we need to show that $(A_{i+1},\Psi_\Sigma(v_i),n_i)\in S_{A_i}$ for all $0\le i\le k$.

Fix an arbitrary index $i$. To deduce that $(A_{i+1},\Psi_\Sigma(v_i),n_i)\in S_{A_i}$, we need to show that 
\begin{center}
(i) $A_{i+1}\cdot [\Psi_\Sigma(v_i)]^{n_i}=A_i$; (ii) $n_i=\mu(A_i)$; and (iii) $\sigma(\Psi_\Sigma(v_i))=\vartheta(A_i)$.
\end{center}

\noindent(i) We have
\begin{align*}
A_{i+1}\cdot [\Psi_\Sigma(v_i)]^{n_i}&=\underbrace{[\Psi_\Sigma(v_{k})]^{n_{k}}[\Psi_\Sigma(v_{k-1})]^{n_{k-1}}\cdots [\Psi_\Sigma(v_{i+1})]^{n_{i+1}}}_{A_{i+1}}\cdot [\Psi_\Sigma(v_i)]^{n_i}\\
&=[\Psi_\Sigma(v_{k})]^{n_{k}}[\Psi_\Sigma(v_{k-1})]^{n_{k-1}}\cdots [\Psi_\Sigma(v_{i})]^{n_{i}}=A_i.
\end{align*}

\noindent(ii) We argue by contradiction. Assume $n_i\neq \mu(A_i)$. By definition, if $n_i>\mu(A_i)$, then $n_i>\max\{\, n\in\mathbb{N}\,\,|\, A_i=A\cdot B^n \text{ for some }A\in\mathcal{P}_{\Sigma} \text{ and } B\in\mathcal{P}_{\Sigma}^+\,\}$. This is a contradiction as $A_i=A_{i+1}\cdot [\Psi_\Sigma(v_i)]^{n_i}$. 

Assume $n_i<\mu(A_i)$. By definition, there exist $A\in\mathcal{P}_\Sigma$ and $B\in\mathcal{P}_\Sigma^+$ such that $A_i=A\cdot B^{\mu(A_i)}$. Choose $u\in\Sigma^*$ and $v\in\Sigma^+$ such that $\Psi_\Sigma(u)=A$ and $\Psi_\Sigma(v)=B$. Thus we have $\Psi_\Sigma(uv^{\mu(A_i)})=A_i=\Psi_\Sigma(v_{k}^{n_{k}}v_{k-1}^{n_{k-1}}\cdots v_i^{n_i})$. By Remark~\ref{RemChopOffWord}, it holds that $\pn_r(v_{k}^{n_{k}}v_{k-1}^{n_{k-1}}\cdots v_i^{n_i})=v_{k}^{n_{k}}v_{k-1}^{n_{k-1}}\cdots v_i^{n_i}$. 

Let $w'=uv^{\mu(A_i)}v_{i-1}^{n_{i-1}}v_{i-2}^{n_{i-2}}\cdots v_0^{n_0}$. Since $uv^{\mu(A_i)}\equiv_Mv_{k}^{n_{k}}v_{k-1}^{n_{k-1}}\cdots v_i^{n_i}$, it follows by the right invariance of $M$\!-equivalence that 
$$w'=\underbrace{uv^{\mu(A_i)}v_{i-1}^{n_{i-1}}v_{i-2}^{n_{i-2}}\cdots v_1^{n_1}}_{u'} v_0^{n_0}\equiv_Mv_{k}^{n_{k}}v_{k-1}^{n_{k-1}}\cdots v_i^{n_i}v_{i-1}^{n_{i-1}}v_{i-2}^{n_{i-2}}\cdots v_1^{n_1}v_0^{n_0}=w.$$
Let $\pn_r(w')=y_{k'}^{m_{k'}}y_{k'-1}^{m_{k'-1}}\cdots y_1^{m_1}y_0^{m_0}$. Since $w'=u'v_0^{n_0}$, by Remark~\ref{RemHighPowWord}, it follows that $m_0\ge n_0$. If $m_0>n_0$, then by Definition~\ref{DefRelationWord}, we have $w\prec w'$ which is a contradiction as $w$ is maximal. Thus $m_0=n_0$. 

\begin{case}$m_0=n_0=1$.\\
Since $m_0=1$, by Definition~\ref{DefNormalFormWord}, it follows that $m_1=\tau(y_{j}^{m_{j}}y_{j-1}^{m_{j-1}}\cdots y_1^{m_1})>1$. Thus, by Definition~\ref{DefNormalFormWord} again, it holds that $|y_0|=\theta(w')\le|v_0|$. 
If $|y_0|<|v_0|$, then by Definition~\ref{DefRelationWord}, we have $w\prec w'$ which is a contradiction. Thus $|y_0|=|v_0|$.
\end{case}

\begin{case}$m_0=n_0>1$.\\
Then, by Definition~\ref{DefNormalFormWord}, it follows that $|y_0|=\theta(w')\ge|v_0|$. If $|y_0|>|v_0|$, then by Definition~\ref{DefRelationWord}, we have $w\prec w'$ which is a contradiction. Thus, $|v_0|=|y_0|$.
\end{case}
In both cases, we have $|v_0|=|y_0|$. Since both $v_0$ and $y_0$ are suffixes of the word $v_i^{n_i}v_{i-1}^{n_{i-1}}v_{i-2}^{n_{i-2}}\cdots v_1^{n_1}v_0^{n_0}$, it follows that $v_0=y_0$. Since $n_0=m_0$ and $v_0=y_0$, by similar argument as above, it can be shown that $m_1=n_1$ and $v_1=y_1$. Arguing continuously like this, we have $y_k=v_k$ and $m_k=n_k$ for all $i-1\le k\le 0$. 

By our assumption, we have $n_i<\mu(A_i)$. Meanwhile by Definition~\ref{DefNormalFormWord}, we have $m_i=\tau(uv^{\mu(A_i)})\ge \mu(A_i)$. Thus, $n_i<\mu(A_i)\le m_i$. By Definition~\ref{DefRelationWord}, it follows that $w\prec w'$ which is a contradiction. Therefore, we conclude that $n_i=\mu(A_i)$. 

\vspace{0.5em}\noindent(iii) Note that since the second diagonal of the Parikh matrix of a word contains the Parikh vector of that word, it follows that $\sigma(\Psi_\Sigma(x))=|x|$ for any $x\in\Sigma^*$. We now argue by contradiction. Assume $\sigma(\Psi_\Sigma(v_i))\neq\vartheta(A_i)$. 

\begin{case2}$\mu(A_i)=n_i=1.$\\
Consider the set
\begin{align*}
\Gamma=\{\,\sigma(B)\,\,|\,\,B\in\mathcal{P}_{\Sigma}^+ \text{ and }
A_i=A\cdot B \text{ for some } A\in\mathcal{P}_\Sigma^+\text{ with }\mu(A)\neq 1\}
\end{align*}
in Definition~\ref{DefNormalForm}.
\begin{subcase2}The set $\Gamma$ is nonempty.\\
By Definition~\ref{DefNormalForm}, it holds that $\vartheta(A_i)$ is the minimum element of the set $\Gamma$.
Therefore, if $\sigma(\Psi_\Sigma(v_i))<\vartheta(A_i)$, then it is a contradiction as $A_i=A_{i+1}\cdot \Psi_\Sigma(v_i)$. 

Assume $\sigma(\Psi_\Sigma(v_i))>\vartheta(A_i)$. Since set $\Gamma$ is nonempty, there exist $A,B\in\mathcal{P}_\Sigma^+$ such that $A_i=A\cdot B$ with $\mu(A)\neq 1$ and $\sigma(B)=\vartheta(A_i)$. Since $\mu(A)\neq 1$, it follows that $A=A'\cdot B'^{n'}$ for some $A'\in\mathcal{P}_\Sigma$, $B'\in\mathcal{P}_\Sigma^+$ and integer $n'>1$.
Choose $u'\in\Sigma^*$ and $v',v\in\Sigma^+$ such that $\Psi_\Sigma(u')=A'$, $\Psi_\Sigma(v')=B'$ and $\Psi_\Sigma(v)=B$. Thus we have $\Psi_\Sigma(u'v'^{n'}v)=A_i=\Psi_\Sigma(v_{k}^{n_{k}}v_{k-1}^{n_{k-1}}\cdots v_i)$.
Note that $|v|<|v_i|$ because
\begin{align*}
|v|=\sigma(\Psi_\Sigma(v))=\sigma(B)=\vartheta(A_i)<\sigma(\Psi_\Sigma(v_i))=|v_i|.
\end{align*}

Let $w'=u'v'^{n'}vv_{i-1}^{n_{i-1}}v_{i-2}^{n_{i-2}}\cdots v_0^{n_0}$. Since $u'v'^{n'}v\equiv_Mv_{k}^{n_{k}}v_{k-1}^{n_{k-1}}\cdots v_{i+1}^{n_{i+1}}v_i$, it follows by the right invariance of $M$\!-equivalence that 
$$w'=u'v'^{n'}vv_{i-1}^{n_{i-1}}v_{i-2}^{n_{i-2}}\cdots v_1^{n_1} v_0^{n_0}\equiv_Mv_{k}^{n_{k}}v_{k-1}^{n_{k-1}}\cdots v_{i+1}^{n_{i+1}}v_iv_{i-1}^{n_{i-1}}v_{i-2}^{n_{i-2}}\cdots v_1^{n_1}v_0^{n_0}=w.$$
Let $\pn_r(w')=y_{k'}^{m_{k'}}y_{k'-1}^{m_{k'-1}}\cdots y_0^{m_0}$. By similar argument as in (ii), it can be shown that $y_j=v_j$ and $m_j=n_j$ for all $0\le j\le i-1$. Now, as for $m_i$, if $m_i>1=n_i$, then $w\prec w'$ by Definition~\ref{DefRelationWord} and thus a contradiction. On the other hand, if $m_i=1$, then since $n'>1$, it follows that $|y_i|=\theta(u'v'^{n'}v)\le |v|$. Since $|v|<|v_i|$, it follows that $|y_i|\le |v|<|v_i|$.
By Definition~\ref{DefRelationWord}, again it follows that $w\prec w'$ which is a contradiction.
\end{subcase2}

\begin{subcase2}The set $\Gamma$ is empty.\\
Note that $A_i=[\Psi_\Sigma(v_{k})]^{n_{k}}[\Psi_\Sigma(v_{k-1})]^{n_{k-1}}\cdots [\Psi_\Sigma(v_i)]^{n_i}$. Since $n_i=1$ and the set $\Gamma$ is empty, it follows that $n_j=1$ for all $i\le j\le k$. Meanwhile, by Remark~\ref{RemChopOffWord}, it holds that $\pn_r(v_{k}^{n_{k}}v_{k-1}^{n_{k-1}}\cdots v_i^{n_i})=v_{k}^{n_{k}}v_{k-1}^{n_{k-1}}\cdots v_i^{n_i}$. Therefore, since $n_j=1$ for all $i\le j\le k$, it must be the case that $i=k$. That is to say, $A_i=\Psi_\Sigma(v_i)$.

Since the set $\Gamma$ is empty, by Definition~\ref{DefNormalForm}, we have $\vartheta(A_i)=\sigma(A_i)$. Then, $\vartheta(A_i)=\sigma(A_i)=\sigma(\Psi_\Sigma(v_i))$, thus a contradiction.
\end{subcase2}
\end{case2}

\begin{case2}$\mu(A_i)=n_i>1.$\\
Then, by Definition~\ref{DefNormalForm}, it holds that $$\vartheta(A_i)=\max\{\,\sigma(B)\,\,|\,\,B\in\mathcal{P}_{\Sigma}^+ \text{ and }
A_i=A\cdot B^{n_i} \text{ for some } A\in\mathcal{P}_\Sigma\}.$$ 
Thus if $\sigma(\Psi_\Sigma(v_i))>\vartheta(A_i)$, then it is a contradiction as $A_i=A_{i+1}\cdot [\Psi_\Sigma(v_i)]^{n_i}$. 

Assume $\sigma(\Psi_\Sigma(v_i))<\vartheta(A_i)$. Let $A\in\mathcal{P}_\Sigma$ and $B\in\mathcal{P}_\Sigma^+$ be such that $A_i=A\cdot B^{n_i}$ with $\sigma(B)=\vartheta(A_i)$. Choose $u\in\Sigma^*$ and $v\in\Sigma^+$ such that $\Psi_\Sigma(u)=A$ and $\Psi_\Sigma(v)=B$. Thus we have $\Psi_\Sigma(uv^{n_i})=A_i=\Psi_\Sigma(v_{k}^{n_{k}}v_{k-1}^{n_{k-1}}\cdots v_i^{n_i})$. By Remark~\ref{RemHighPowWord}, it holds that $\tau(uv^{n_i})\ge n_i$. Assume $\tau(uv^{n_i})>n_i$. Then we have $uv^{n_i}=u'v'^{n'}$ for some $u\in\Sigma^*$ and $v\in\Sigma^+$ where $n'=\tau(uv^{n_i})$. However note that $A_i=\Psi_\Sigma(u')\cdot [\Psi_\Sigma(v')]^{n'}$ and $n'=\tau(uv^{n_i})> n_i=\mu(A_i)$. This is a contradiction by the definition of $\mu(A_i)$. Thus $\tau(uv^{n_i})=n_i$.

Let $w'=uv^{n_i}v_{i-1}^{n_{i-1}}v_{i-2}^{n_{i-2}}\cdots v_0^{n_0}$. Since $uv^{n_i}\equiv_Mv_{k}^{n_{k}}v_{k-1}^{n_{k-1}}\cdots v_i^{n_i}$, it follows by the right invariance of $M$\!-equivalence that 
$$w'=uv^{n_i}v_{i-1}^{n_{i-1}}v_{i-2}^{n_{i-2}}\cdots v_1^{n_1} v_0^{n_0}\equiv_Mv_{k}^{n_{k}}v_{k-1}^{n_{k-1}}\cdots v_i^{n_i}v_{i-1}^{n_{i-1}}v_{i-2}^{n_{i-2}}\cdots v_1^{n_1}v_0^{n_0}=w.$$

Let $\pn_r(w')=y_{k'}^{m_{k'}}y_{k'-1}^{m_{k'-1}}\cdots y_0^{m_0}$. By similar argument as in (ii), it can be shown that $y_k=v_k$ and $m_k=n_k$ for all $0\le k\le i-1$. Furthermore, we have  $\tau(uv^{n_i})=n_i=\tau(v_{k}^{n_{k}}v_{k-1}^{n_{k-1}}\cdots v_i^{n_i})$ and $|v|>|v_i|$. Thus by Definition~\ref{DefRelationWord}, it follows that $w\prec w'$ which is a contradiction.
\end{case2}
In both cases, we obtain a contradiction. Thus $\sigma(\Psi_\Sigma(v_i))=\vartheta(A_i)$. Since (i), (ii) and (iii) hold, our conclusion follows.
\end{proof}

\begin{example}
Suppose $\Sigma=\{a<b\}$. Consider the Parikh matrix $M$ stated in Example~\ref{ExNormalMat}. That matrix $M$ represents the following $M$\!-equivalent words:
\begin{align*}
&aaaaabbabaa,\,\,aaaabaabbaa,\,\,aaaababaaba,\,\,aaaabbaaaab,\,\,aaabaaababa,\\
&aaabaabaaab,\,\,aabaaaaabba,\,\,aabaaaabaab,\,\,abaaaaaabab,\,\,baaaaaaaabb.
\end{align*}
Rewriting the above words in their respective Parikh normal forms, we have 
\vspace{0.5em}\begin{center}
\begin{tabular}{ccccc}
$a^5b^2aba^2$,  & $a^4ba^2b^2a^2$, & $a^4b(aba)^2$,    & $a^4b^2a^4b$,   & $a^3ba^3(ba)^2$\\
$a(aab)^2a^3b$, & $a^2ba^5b^2a$,   & $a^2ba^2(aab)^2$, & $aba^5(ab)^2$,  & $ba^8b^2$.
\end{tabular}
\end{center}
\vspace{0.5em}Notice that $a^4b(aba)^2$ and $a^2ba^2(aab)^2$ are the only $\prec$-maximal words. Correspondingly, the only \textit{rl}-Parikh normal forms of matrix $M$ are 
$$[\Psi_\Sigma(a)]^{4}[\Psi_\Sigma(b)][\Psi_\Sigma(aba)]^{2} \text{ and } [\Psi_\Sigma(a)]^{2}[\Psi_\Sigma(b)][\Psi_\Sigma(a)]^{2}[\Psi_\Sigma(aab)]^{2},$$
which are in fact the matrices (1) and (2) in Example~\ref{ExNormalMat}.
\end{example}

The following is the converse of Theorem~\ref{RelNormalMatWord}.

\begin{theorem}\label{RelNormalMatWord2}
Suppose $\Sigma$ is an ordered alphabet with $|\Sigma|=s$ and $M\in\mathcal{P}_\Sigma$. Assume $B_{k}^{n_{k}}B_{k-1}^{n_{k-1}}\cdots B_0^{n_0}$ is an \textit{rl}-Parikh normal form of $M$. Suppose $w\in\Sigma^*$ such that $w=v_{k}^{n_{k}}v_{k-1}^{n_{k-1}}\cdots v_0^{n_0}$ where for every integer $0\le i\le k$, we have $v_i\in\Sigma^+$ with $\Psi_\Sigma(v_i)=B_i$. Then, $\pn_r(w)=v_{k}^{n_{k}}v_{k-1}^{n_{k-1}}\cdots v_0^{n_0}$ and $w$ is $\prec$-maximal. 
\end{theorem}

\begin{proof}
It can be shown that $\pn_r(w)=v_{k}^{n_{k}}v_{k-1}^{n_{k-1}}\cdots v_0^{n_0}$ and $w$ is $\prec$-maximal by referring to Definition~\ref{DefNormalFormWord}, Definition~\ref{DefNormalForm}, Remark~\ref{RemChopOffMat}, Remark~\ref{RemHighPowMat} and arguing analogously to the proof of Theorem~\ref{RelNormalMatWord}. However, the explicit proof of this theorem is not presented here as it resembles that of Theorem~\ref{RelNormalMatWord}.
\end{proof}

\section{Conclusion}
We have seen that Parikh matrices are versatile in the study of subword occurrences in words which are in the form of powers. In fact, by using Theorem~\ref{PowerTheorem}, one can acquire information on the subword occurrences in arbitrary power of any word by just knowing the base word.

Definition~\ref{DefNormalForm} and Definition~\ref{DefNormalFormWord} can be modified in a way such that the decompositions commence from left to right. Accordingly, one could term the corresponding forms obtained as the $lr$-Parikh normal forms. For both Parikh matrices and words, it can then be studied to what extent the $rl$-Parikh normal forms and $lr$-Parikh normal forms are related to each other.

Last but not least, Proposition~\ref{EqualityMequivClass} is an interesting observation on the study of \mbox{$M$\!-equivalence} of powers of words, which we would further investigate in our future contribution. 
For $\Sigma=\{a<b<c\}$, we see that there exists $w\in\Sigma^*$ satisfying the equality $|C_{w^2}|=|C_w|^2$ for arbitrary $|C_w|=N$. For the case $N=1$, consider the word $w=abcb$ while for the case $N>1$, consider the word $w=a^{N-1}cb$ (notice that $|C_w|=N$). In both cases, we have $|C_{w^2}|=|C_w|^2$. Thus it is intriguing to know whether the following general result holds:
\begin{center}
\textit{Suppose $\Sigma=\{a<b<c\}$ and $w\in\Sigma^*$. For any positive integer $m$, there exists $w\in\Sigma^*$ satisfying the equality $|C_{w^m}|=|C_w|^m$ for arbitrary $|C_w|$.}
\end{center}

\section*{Acknowledgement}
The second and third authors gratefully acknowledge support for this research by a Research University Grant No.~1011/PMATHS/8011019 of Universiti Sains Malaysia. This paper is a part of the second author's Ph.D work.

\bibliographystyle{abbrv}

\end{document}